\pdfoutput=1
\RequirePackage{ifpdf}
\ifpdf 
\documentclass[pdftex]{sigma}
\else
\documentclass{sigma}
\fi

\numberwithin{equation}{section}

\newtheorem{Theorem}{Theorem}[section]

\newtheorem{lem}[Theorem]{Lemma}
\newtheorem{prop}[Theorem]{Proposition}
 { \theoremstyle{definition}
\newtheorem{dfn}[Theorem]{Definition}

\newtheorem{rem}[Theorem]{Remark} }
\newtheorem{dfnprop}[Theorem]{Definition and Proposition}

\usepackage{mathrsfs}

\newcommand{\cd}{\cdot}
\newcommand{\hot}{\widehat \otimes}
\newcommand{\op}{\oplus}
\newcommand{\ti}{\times}
\newcommand{\nn}{\mathbb{N}}
\newcommand{\zz}{\mathbb{Z}}
\newcommand{\rr}{\mathbb{R}}

\newcommand{\be}{\beta}
\newcommand{\ga}{\gamma}
\newcommand{\De}{\Delta}
\newcommand{\io}{\iota}
\newcommand{\la}{\lambda}
\newcommand{\si}{\sigma}
\newcommand{\ov}{\overline}
\newcommand{\C}[1]{\mathcal{#1}}
\newcommand{\s}[1]{\mathscr{#1}}
\newcommand{\B}[1]{\mathbb{#1}}
\newcommand{\su}{\subseteq}
\newcommand{\q}{\qquad}
\newcommand{\wit}{\widetilde}
\newcommand{\inn}[1]{\langle #1 \rangle}
\newcommand{\sem}{\setminus}
\newcommand\sA{\mathscr{A}}

\begin{document}
\allowdisplaybreaks

\newcommand{\arXivNumber}{1901.05161}

\renewcommand{\thefootnote}{}

\renewcommand{\PaperNumber}{082}

\FirstPageHeading

\ShortArticleName{On the Unbounded Picture of $KK$-Theory}

\ArticleName{On the Unbounded Picture of $\boldsymbol{KK}$-Theory\footnote{This paper is a~contribution to the Special Issue on Noncommutative Manifolds and their Symmetries in honour of Giovanni Landi. The full collection is available at \href{https://www.emis.de/journals/SIGMA/Landi.html}{https://www.emis.de/journals/SIGMA/Landi.html}}}

\Author{Jens KAAD}

\AuthorNameForHeading{J.~Kaad}

\Address{Department of Mathematics and Computer Science, The University of Southern Denmark,\\ Campusvej 55, DK-5230 Odense M, Denmark}
\Email{\href{mailto:kaad@imada.sdu.dk}{kaad@imada.sdu.dk}}
\URLaddress{\url{https://portal.findresearcher.sdu.dk/en/persons/kaad}}

\ArticleDates{Received October 22, 2019, in final form August 05, 2020; Published online August 22, 2020}

\Abstract{In the founding paper on unbounded $KK$-theory it was established by Baaj and Julg that the bounded transform, which associates a class in $KK$-theory to any unbounded Kasparov module, is a surjective homomorphism (under a separability assumption). In this paper, we provide an equivalence relation on unbounded Kasparov modules and we thereby describe the kernel of the bounded transform. This allows us to introduce a notion of topological unbounded $KK$-theory, which becomes isomorphic to $KK$-theory via the bounded transform. The equivalence relation is formulated entirely at the level of unbounded Kasparov modules and consists of homotopies together with an extra degeneracy condition. Our degenerate unbounded Kasparov modules are called spectrally decomposable since they admit a decomposition into a part with positive spectrum and a part with negative spectrum.}


\Classification{19K35; 58B34}

\begin{flushright}
\begin{minipage}{70mm}
\it Dedicated to Gianni Landi on the occasion\\ of his 60th birthday
\end{minipage}
\end{flushright}

\renewcommand{\thefootnote}{\arabic{footnote}}
\setcounter{footnote}{0}

\section{Introduction}
$KK$-theory, as introduced by Kasparov in \cite{Kas:TIE,Kas:OFE}, has its roots in the Brown--Douglas--Fillmore extension theory of commutative $C^*$-algebras~\cite{BDF:ECK}, and in Atiyah's axiomatization of properties of elliptic operators on manifolds, \cite{Ati:GEO}. But $KK$-theory extends far beyond the context of commutative $C^*$-algebras and has become an important tool for accessing the algebraic topology of $C^*$-algebras with applications ranging from Elliott's classification program to key aspects of index theory.

In practice, explicit classes in $KK$-theory often come from unbounded operators acting on Hilbert $C^*$-modules and these unbounded operators constitute the main ingredient in the unbounded picture of $KK$-theory. The cycles in the unbounded picture are called unbounded Kasparov modules and are often of a differential geometric origin with prototypical examples being Dirac operators (in the case of $K$-homology) or multiplication operators by symbols of Dirac operator (in the case of $K$-theory). In the unbounded picture, the relationship between $KK$-theory and the program of Connes on noncommutative geometry is in fact immediate: spectral triples are, without any further modifications, examples of unbounded Kasparov modules~\cite{Con:NCG,Con:GCM}.

The passage from the unbounded picture to the more commonly encountered bounded picture of $KK$-theory is furnished by the bounded transform which turns an unbounded Kasparov module into a class in $KK$-theory via the smooth approximation of the sign function given by $t \mapsto t\big(1 + t^2\big)^{-1/2}$ (and the functional calculus on Hilbert $C^*$-modules). The richness of the unbounded picture is witnessed by a fundamental theorem of Baaj and Julg stating that any class in $KK$-theory comes from an unbounded Kasparov module so that the bounded transform is in fact a surjection (under a mild separability condition)~\cite{BaJu:TBK}. See also \cite{Kuc:LIK,MeRe:NST} for other interesting and related lifting results.

In this paper, we construct an equivalence relation on unbounded Kasparov modules which captures the kernel of the bounded transform and this equivalence relation is ``geometric'' in the sense that it can be formulated without any reference to the bounded picture of $KK$-theory. Our equivalence relation relies on an extra degeneracy condition on unbounded Kasparov modules together with a notion of homotopies using families of unbounded Kasparov modules parametrized by the unit interval. Our degeneracy condition on an unbounded Kasparov module is formulated in terms of a spectral decomposition of the unbounded operator in question building on the simple observation that the phase of an unbounded selfadjoint and regular operator with strictly positive spectrum is equal to the identity operator. It was pointed out to us by the referee that spectral decomposability is related to the concept of weak degeneracy from \cite[Definition~3.1]{DGM:BGU} and that it can be formulated alternatively as a condition on the bounded transform. We are of course grateful for these comments.

In summary, we introduce a notion of topological unbounded $KK$-theory and show that topological unbounded $KK$-theory is isomorphic to $KK$-theory via the bounded transform. We hereby give an affirmative answer to the question raised by Deeley, Goffeng, and Mesland on page $3$ in~\cite{DGM:BGU}. In fact, it turns out that our topological unbounded $KK$-theory is independent of the choice of a dense $*$-subalgebra of a $C^*$-algebra as long as this $*$-subalgebra is countably generated. It would thus be interesting to investigate the relationship between topological unbounded $KK$-theory and the bordism group introduced in~\cite{DGM:BGU}, where the equivalence relation comes from Hilsum's notion of bordisms of unbounded Kasparov modules~\cite{Hil:BIK}.

The idea for proving the injectivity of the bounded transform is to apply the lifting procedure introduced by Baaj and Julg to a homotopy at the bounded level, where the homotopic elements are bounded transforms of some given unbounded Kasparov modules. The problem with this idea is that it might very well happen that the unbounded homotopy achieved from this process connects two unbounded Kasparov module that are very different from the original ones. To see what might happen, notice that an unbounded Kasparov module could satisfy that commutators extend to \emph{bounded} operators for all elements in a dense subalgebra whereas a Baaj--Julg lift can always be chosen such that commutators extend to \emph{compact} operators for all elements in a (perhaps different) dense subalgebra. In this paper, we resolve this problem by studying the concept of a spectrally decomposable unbounded Kasparov module which provides a type of degenerate elements at the unbounded level, related to the idea that a strictly positive unbounded operator should not contain any topological information.

After this paper was written and made available on the arXiv a strengthening of our results was obtained by van den Dungen and Mesland~\cite{DuMe:HEU}. Among other things these authors were able to prove that a spectrally decomposable unbounded Kasparov module is in fact null-homotopic at the unbounded level, see~\cite[Corollary~4.9]{DuMe:HEU}. Notably, using ideas related to \cite{BaJu:TBK,Kuc:LIK,MeRe:NST}, van den Dungen and Mesland also establish a lifting result regarding homotopies of Kasparov modules which is stronger than the lifting results applied in the present text, see~\cite[Theorem~2.9]{DuMe:HEU}.

We emphasize that the word topological is a keyword in connection with our definition of unbounded $KK$-theory. In other approaches to unbounded $KK$-theory, the aim is to find an interesting equivalence relation which captures geometric content at the level of unbounded Kasparov modules while still admitting explicit formulae for the interior Kasparov product. The geometric content which could be valuable in this respect relates to the asymptotic behaviour of eigenvalues and the spectral metric aspects of noncommutative geometry. Certainly, this geometric content is invisible from a topological point of view and thus in particular from the point of view of topological unbounded $KK$-theory. The delicate questions on the geometric nature of unbounded $KK$-theory are part of ongoing research on the unbounded Kasparov product and the interested reader can consult the following (incomplete list of) references: \cite{BrMeSu:GSU,Kaa:MEU,KaLe:SFU,KaSu:FDT, Kuc:PUM,Kuc:LIK,Mes:UCN,MeRe:NST}.

\subsection{Standing assumptions} Throughout this text $A = A_0 \op A_1$ and $B = B_0 \op B_1$ will be $\zz/2\zz$-graded $C^*$-algebras with $A$ separable and $B$ $\si$-unital (meaning that $B$ has a countable approximate identity). We moreover fix a norm-dense $\zz/2\zz$-graded $*$-subalgebra $\s A \su A$, which we require to be generated as a~$*$-algebra by some countable subset $\{ x_j \,|\, j \in \nn \} \su \s A$. Remark that the grading on $\s A$ is compatible with the grading on $A$ meaning that the $\zz/2\zz$-grading operator $\ga_A \colon A \to A$ (being $1$ on~$A_0$ and~$-1$ on~$A_1$) induces the $\zz/2\zz$-grading operator $\ga_{\s A} \colon \s A \to \s A$.

\section[Kasparov modules and $KK$-theory]{Kasparov modules and $\boldsymbol{KK}$-theory}\label{s:boukas}
In this section we give a brief summary of the main definitions concerning Kasparov modules and $KK$-theory. For more details the reader can consult the following references: \cite{Bla:KOA,JeTh:EKT, Kas:OFE}. The commutators appearing in this section are all graded commutators. For a $\zz/2\zz$-graded $C^*$-correspondence $X$ from $A$ to $B$ we usually suppress the even $*$-homomorphism $\pi_X\colon A \to \B L(X)$, which determines the left action of $A$ on $X$ (where $\B L(X)$ denotes the $\zz/2\zz$-graded $C^*$-algebra of bounded adjointable operators on $X$).

\begin{dfn}
A \emph{Kasparov module} from $A$ to $B$ is a pair $(X,F)$ where $X$ is a countably generated $\zz/2\zz$-graded $C^*$-correspondence from $A$ to $B$ and $F\colon X \to X$ is an odd bounded adjointable operator such that
\[
a (F - F^*) , \ a \big(F^2 - 1\big) , \ [F,a] \colon \ X \to X
\]
are compact operators for all $a \in A$.

A Kasparov module $(X,F)$ from $A$ to $B$ is \emph{degenerate} when
\[
a (F - F^*) = a \big(F^2 - 1\big) = [F,a] = 0 \qquad \mbox{for all} \quad a \in A .
\]
\end{dfn}

\begin{dfn}Two Kasparov modules $(X_0,F_0)$ and $(X_1,F_1)$ from $A$ to $B$ are \emph{unitarily equivalent} when there exists an even unitary isomorphism of $\zz/2\zz$-graded $C^*$-correspondences $U\colon X_0 \to X_1$ such that $U F_0 U^* = F_1$. In this case we write $(X_0,F_0) \sim_u (X_1,F_1)$. Remark that~$U$ (by definition) has to intertwine the left actions as well so that $U \pi_{X_0}(a) U^* = \pi_{X_1}(a)$ for all $a \in A$.
\end{dfn}

When given a $\zz/2\zz$-graded $C^*$-algebra $C$ and an even $*$-homomorphism $\be\colon B \to C$ we may ``change the base'' of a $\zz/2\zz$-graded $C^*$-correspondence $X$ from $A$ to $B$. Indeed, we may consider~$C$ as a~$\zz/2\zz$-graded $C^*$-correspondence from $B$ to $C$ and form the interior tensor product $X \hot_{\be} C$ which is a $\zz/2\zz$-graded $C^*$-correspondence from~$A$ to~$C$. Any bounded adjointable operator $T\colon X \to X$ then induces a bounded adjointable operator $T \hot_{\be} 1\colon X \hot_{\be} C \to X \hot_{\be} C$ and this operation yields an even $*$-homomorphism $\B L(X) \to \B L\big(X \hot_{\be} C\big)$, see for example \cite[Chapter~4]{Lan:HCM}.

\begin{dfn}Two Kasparov modules $(X_0,F_0)$ and $(X_1,F_1)$ both from $A$ to $B$ are \emph{homotopic} when there exists a Kasparov module $(X,F)$ from $A$ to $C([0,1],B)$ such that
\[
\big(X \hot_{{\rm ev}_0} B,F \hot_{{\rm ev}_0} 1\big) \sim_u (X_0,F_0) \qquad \mbox{and} \qquad \big(X \hot_{{\rm ev}_1} B,F \hot_{{\rm ev}_1} 1\big) \sim_u (X_1,F_1) ,
\]
where ${\rm ev}_t\colon C( [0,1],B) \to B$ denotes the even $*$-homomorphism given by evaluation at $t \in [0,1]$. In this case we write $(X_0,F_0) \sim_h (X_1,F_1)$.
\end{dfn}

It is a non-trivial fact that the above homotopy relation is an equivalence relation and it can be difficult to find a record of this result in the standard literature on $KK$-theory. We state the result as a proposition here and notice that the proof is very similar to the proof given in the unbounded setting later on, see Proposition~\ref{p:unbhom} and in particular Lemma~\ref{l:coucom} which can be applied to prove the transitivity of the relation $\sim_h$.

\begin{prop}Homotopy of Kasparov modules is an equivalence relation.
\end{prop}

\begin{dfn}\emph{$KK$-theory} from $A$ to $B$ consists of Kasparov modules from $A$ to $B$ modulo homotopies. $KK$-theory from $A$ to $B$ is denoted by $KK(A,B)$.
\end{dfn}

We may form the \emph{direct sum} of two Kasparov modules $(X,F)$ and $(X',F')$ from $A$ to $B$ and this is the Kasparov module from $A$ to $B$ given by
\[
(X,F) + (X',F') := (X \op X', F \op F') .
\]
The \emph{zero module} is the Kasparov module $(0,0)$ from $A$ to $B$.

We quote the following two results from \cite[Chapter~17]{Bla:KOA}:

\begin{prop}\label{p:degbou}Any degenerate Kasparov module from $A$ to $B$ is homotopic to the zero module.
\end{prop}

\begin{Theorem}The direct sum operation and the zero module provide $KK$-theory from~$A$ to~$B$ with the structure of an abelian group.
\end{Theorem}

\section{Unbounded Kasparov modules}
In this section we review the main results of the paper \cite{BaJu:TBK}, which can be regarded as the founding paper on unbounded $KK$-theory. We recall that a symmetric unbounded operator $D\colon \operatorname{Dom}(D) \to X$ acting on a Hilbert $C^*$-module $X$ over $B$ is selfadjoint and regular when the operators $D \pm i\colon \operatorname{Dom}(D) \to X$ are surjective, see \cite[Lemmas~9.7 and~9.8]{Lan:HCM}. Unbounded selfadjoint and regular operators admit a continuous functional calculus as developed in~\cite{Wor:UAQ,WoNa:OTC}, see also \cite[Theorem~10.9]{Lan:HCM}. Notice that in our convention all unbounded operators are densely defined.

\begin{dfn}\label{d:unbkas}
An \emph{unbounded Kasparov module} from $\s A$ to $B$ is a pair $(X,D)$, where $X$ is a~countably generated $\zz/2\zz$-graded $C^*$-correspondence from $A$ to $B$ and $D \colon \operatorname{Dom}(D) \to X$ is an odd unbounded selfadjoint and regular operator such that
\begin{enumerate}\itemsep=0pt
\item[1)] each $a \in \s A$ preserves the domain of $D$ and the graded commutator $[D,a]\colon \operatorname{Dom}(D) \to X$ extends to a bounded operator $d(a)\colon X \to X$;
\item[2)] the operator $a \cd (i + D)^{-1}\colon X \to X$ is compact for all $a \in \s A$.
\end{enumerate}

An unbounded Kasparov module $(X,D)$ from $\s A$ to $B$ is \emph{Lipschitz regular} when the graded commutator
\[
[|D|,a] \colon \ \operatorname{Dom}(D) \to X
\]
extends to a bounded operator on $X$ for all $a \in \s A$.
\end{dfn}

For an unbounded Kasparov module $(X,D)$ it follows automatically that $d(a)\colon X \to X$ is adjointable for all $a \in \s A$ and we have the formulae
\[
d(a)^* = \begin{cases} - d(a^*) & \text{for} \ a \text{ even}, \\ d(a^*) & \text{for} \ a \text{ odd}.\end{cases}
\]

\begin{dfn}\label{d:dirsum}
The \emph{direct sum} of two unbounded Kasparov modules $(X,D)$ and $(X', D')$ from $\s A$ to $B$ is the unbounded Kasparov module
\[
(X, D) + (X', D') := (X \op X', D \op D')
\]
from $\s A$ to $B$. The \emph{zero module} from $\s A$ to $B$ is the unbounded Kasparov module $(0,0)$.
\end{dfn}

It was proved in \cite{BaJu:TBK} that every unbounded Kasparov module represents a class in $KK$-theory:

\begin{Theorem}
Suppose that $(X,D)$ is an unbounded Kasparov module from $\s A$ to $B$. Then the pair $\big(X, D\big(1 + D^2\big)^{-1/2}\big)$ is a Kasparov module from $A$ to $B$. In particular, we have an associated class $\big[X, D\big(1 + D^2\big)^{-1/2}\big] \in KK(A,B)$ in $KK$-theory.
\end{Theorem}

We refer to the assignment $(X,D) \mapsto \big[ X, D\big(1 + D^2\big)^{-1/2}\big]$ which sends an unbounded Kasparov module from $\s A$ to $B$ to its associated class in $KK(A,B)$ as the \emph{Baaj--Julg bounded transform}.

It turns out that every class in $KK$-theory can be represented by an unbounded Kasparov module. This result is also due to Baaj and Julg~\cite{BaJu:TBK}. The standing hypothesis that $\s A \su A$ is countably generated as a~$*$-algebra plays a crucial role in the proof.

Notice that a bounded positive operator $\De\colon X \to X$ is strictly positive precisely when the image of $\De\colon X \to X$ is dense in $X$ and in this case $\De^{-1}\colon \operatorname{Im}(\De) \to X$ is an unbounded positive and regular operator, see \cite[Lemma~10.1]{Lan:HCM}. In particular, $\operatorname{Dom}\big(\De^{-1}\big) := \operatorname{Im}(\De)$.

\begin{Theorem}\label{t:bousur}
Suppose that $\s A \su A$ is a norm-dense and countably generated $\zz/2\zz$-graded $*$-subalgebra. Suppose moreover that $(X,F)$ is a Kasparov module from $A$ to $B$ with $F = F^*$ and $F^2 = 1$. Then there exists an even strictly positive compact operator $\De\colon X \to X$ such that
\begin{enumerate}\itemsep=0pt
\item[$1)$] the operator $F$ preserves the domain of $\De^{-1}$ and $\big[ F, \De^{-1} \big] = 0$ on $\operatorname{Dom}\big(\De^{-1}\big)$;
\item[$2)$] each $a \in \s A$ preserves the domain of $\De^{-1}$ and $\big[a,\De^{-1}\big] \colon \operatorname{Dom}\big(\De^{-1}\big) \to X$ extends to a~compact operator on~$X$;
\item[$3)$] for each $a \in \s A$, the image of the graded commutator $[F,a]$ is contained in $\operatorname{Dom}\big(\De^{-1}\big)$ and $\De^{-1}[F,a]\colon X \to X$ is a compact operator.
\end{enumerate}
Moreover, with $D := \De^{-1} F\colon \operatorname{Dom}\big(\De^{-1}\big) \to X$ we have that $(X, D)$ is an unbounded Kasparov module from $\s A$ to~$B$ satisfying that
\[
\big[ X, D\big(1 + D^2\big)^{-1/2}\big] = [X,F]
\]
in $KK(A,B)$. In particular, it holds that the Baaj--Julg bounded transform is surjective.
\end{Theorem}

In the context of the above theorem, it is worthwhile to notice that the graded commutator $\big[\De^{-1} F,a\big]\colon \operatorname{Dom}\big(\De^{-1}\big) \to X$ does in fact extend to a \emph{compact} operator for all $a \in \s A$ and that $\big(i + \De^{-1} F\big)^{-1}\colon X \to X$ is compact even though the separable $C^*$-algebra $A$ need not be unital.

\section{Equivalence relations on unbounded Kasparov modules}
In this section we introduce an equivalence relation on unbounded Kasparov modules and use this equivalence relation to construct the topological unbounded $KK$-theory. A key ingredient in our approach is the following notion of a degenerate cycle:

\begin{dfn}\label{d:spedec}
An unbounded Kasparov module $(X,D)$ from $\s A$ to $B$ is \emph{spectrally decomposable} when there exists an orthogonal projection $P \colon X \to X$ such that
\begin{enumerate}\itemsep=0pt
\item[1)] $P$ preserves the domain of $D$ and $D P - P D = 0$ on $\operatorname{Dom}(D)$;
\item[2)] $D P$ and $D (P - 1)$ are unbounded positive operators;
\item[3)] $a P = P a$ for all even elements $a \in A$ and $a P = (1 - P) a$ for all odd elements $a \in A$;
\item[4)] $\ga P = (1 - P) \ga$, where $\ga\colon X \to X$ is the $\zz/2\zz$-grading operator on~$X$.
\end{enumerate}
\end{dfn}

Notice that it follows from Definition~\ref{d:spedec}$(1)$ and the regularity and selfadjointness of~$D$ that $DP$ and $D(P-1)$ are selfadjoint and regular as well.

For a spectrally decomposable unbounded Kasparov module $(X,D)$ with spectral decomposition given by an orthogonal projection $P\colon X \to X$, we apply the notation
\[
D_+ := DP \colon \ \operatorname{Dom}(DP) \to X \qquad \text{and} \qquad D_- := D(P - 1)\colon \ \operatorname{Dom}(D(P-1)) \to X
\]
for the associated unbounded positive and regular operators. 

We remark that $\operatorname{Dom}(D_+) = \{ \xi \in X \,|\, P \xi \in \operatorname{Dom}(D) \}$ and $\operatorname{Dom}(D_-) = \{ \xi \in X \,|\, (P - 1) \xi \in \operatorname{Dom}(D) \}$ so that we obtain the decomposition:
\[
\operatorname{Dom}(D) = \operatorname{Dom}(D_+) \cap \operatorname{Dom}(D_-), \qquad D = D_+ - D_- .
\]
We refer to \cite[Section~6]{Kaa:DAH} for more information on products of unbounded selfadjoint and regular operators with bounded adjointable operators.

\begin{lem}\label{l:spezer}
Suppose that an unbounded Kasparov module $(X,D)$ from $\s A$ to $B$ is spectrally decomposable. Then the class $\big[X,D\big(1 + D^2\big)^{-1/2}\big] \in KK(A,B)$ is equal to zero.
\end{lem}
\begin{proof}By $(3)$ and $(4)$ in Definition~\ref{d:spedec}, the pair $(X, 2 P - 1)$ is a degenerate Kasparov module from $A$ to $B$ and by Proposition~\ref{p:degbou} it therefore suffices to show that $\big(X,D\big(1 + D^2\big)^{-1/2}\big)$ is homotopic to $(X, 2P - 1)$. We let $a \in A$ and show that $D\big(1 + D^2\big)^{-1/2} P a - P a$ and $D\big(1 + D^2\big)^{-1/2}(1 - P) a - (P - 1) a$ are compact operators on $X$. By~$(1)$ in Definition~\ref{d:spedec}, it holds that
\[
1 + D_+^2 = \big(1 + D^2\big) P + 1 - P , \qquad 1 + D_-^2 = \big(1 + D^2\big)(1 - P) + P ,
\]
implying the identities
\[
\big(1 + D_+^2 \big)^{-1/2}P = \big(1 + D^2\big)^{-1/2} P , \qquad \big(1 + D_-^2\big)^{-1/2}(1-P) = \big(1 + D^2\big)^{-1/2} (1 - P) .
\]
We thus conclude that
\begin{gather*}
D\big(1 + D^2\big)^{-1/2} P = D_+ \big(1 + D_+^2\big)^{-1/2} , \\ D \big(1 + D^2\big)^{-1/2} (1 - P) = - D_- \big(1 + D_-^2 \big)^{-1/2} .
\end{gather*}
Using $(2)$ in Definition~\ref{d:spedec}, we see that $D_+ - \big(1 + D_+^2\big)^{1/2}\colon \operatorname{Dom}(D_+) \to X$ and $D_- - \big(1 + D_-^2 \big)^{1/2}\colon \operatorname{Dom}( D_-) \to X$ extend to bounded adjointable operators on~$X$. But this implies that
\begin{gather*}
 D\big(1 + D^2\big)^{-1/2} P a - P a = \big(D_+ - \big(1 + D_+^2\big)^{1/2}\big)\big(1 + D_+^2\big)^{-1/2} P a \\
\hphantom{D\big(1 + D^2\big)^{-1/2} P a - P a}{} = \big(D_+ - \big(1 + D_+^2\big)^{1/2}\big) \big(1 + D^2\big)^{-1/2} P a
\end{gather*}
and
\begin{gather*}
D\big(1 + D^2\big)^{-1/2} (1 - P) a - (P-1) a = - \big(D_- - \big(1 + D_-^2)^{1/2}\big)\big(1 + D_-^2\big)^{-1/2} (1 - P) a \\
\hphantom{D\big(1 + D^2\big)^{-1/2} (1 - P) a - (P-1) a}{} = -\big(D_- - \big(1 + D_-^2\big)^{1/2}\big) \big(1 + D^2\big)^{-1/2} (1 - P) a
\end{gather*}
are compact operators on $X$.
\end{proof}

\begin{rem}As pointed out to us by the referee, an alternative proof of Lemma \ref{l:spezer} can be given using a result of Skandalis~\cite[Lemma~11]{Ska:SRK}. Indeed, with notation as in Lemma~\ref{l:spezer} we obtain that
\[
(2P-1) D\big(1 + D^2\big)^{-1/2} + D\big(1 + D^2\big)^{-1/2} (2P-1) = 2(D_+ + D_-)\big(1 + D^2\big)^{-1/2} \geq 0 .
\]
It then follows from \cite[Lemma~11]{Ska:SRK} that the Kasparov modules $\big(X,D\big(1 + D^2\big)^{-1/2}\big)$ and \mbox{$(X,2P-1)$} are operator homotopic and hence also homotopic.

In fact, as was also remarked by the referee, even more is true: an unbounded Kasparov module $(X,D)$ from $\s A$ to $B$ is spectrally decomposable if and only if there exists a selfadjoint unitary operator $F\colon X \to X$ satisfying that $(X,F)$ is a degenerate Kasparov module together with the conditions $F D\big(1+D^2\big)^{-1/2} - D\big(1+D^2\big)^{-1/2} F = 0$ and $F D\big(1+D^2\big)^{-1/2} + D\big(1+D^2\big)^{-1/2} F \geq 0$. As such, spectral decomposability can be viewed as a condition on the bounded transform of $(X,D)$.
\end{rem}

\begin{dfn}Two unbounded Kasparov modules $(X,D)$ and $(X',D')$ are \emph{unitarily equivalent} when there exists an even unitary isomorphism of $C^*$-correspondences $U\colon X \to X'$ such that $U D U^* = D'$.
\end{dfn}

Unitary equivalence of unbounded Kasparov modules is indeed an equivalence relation and we denote it by $\sim_u$.

Suppose now that $C$ is an extra $\zz/2\zz$-graded $\si$-unital $C^*$-algebra and that $\be\colon B \to C$ is an even $*$-homomorphism. As in Section~\ref{s:boukas} we have the change of base operation given by the interior tensor product of $\zz/2\zz$-graded $C^*$-correspondences: $X \hot_\be C$. Moreover, any unbounded selfadjoint and regular operator $D\colon \operatorname{Dom}(D) \to X$ induces an unbounded selfadjoint and regular operator $D \hot_\be 1\colon \operatorname{Dom}\big(D \hot_\be 1\big) \to X \hot_\be C$, which has resolvents given by
\[
\big( D \hot_\be 1 + i \la \big)^{-1} = (D + i \la)^{-1} \hot_\be 1 \qquad \text{for all} \quad \la \in \rr \sem \{0\} .
\]
In particular, if $a (i + D)^{-1}\colon X \to X$ is a compact operator for some $a \in A$, then
\[
a \big(D \hot_\be 1 + i\big)^{-1} = a (D + i)^{-1} \hot_\be 1\colon \ X \hot_\be C \to X \hot_\be C
\]
is a compact operator as well, see \cite[Proposition~4.7]{Lan:HCM}. These observations allow us to formulate our notion of homotopy at the level of unbounded Kasparov modules:

\begin{dfn}
Two unbounded Kasparov modules $(X_0,D_0)$ and $(X_1,D_1)$ both from $\s A$ to $B$ are \emph{homotopic} when there exists an unbounded Kasparov module $(X,D)$ from $\s A$ to $C([0,1],B)$ such that
\[
\big(X \hot_{{\rm ev}_0} B,D \hot_{{\rm ev}_0} 1\big) \sim_u (X_0,D_0) \qquad \mbox{and} \qquad \big(X \hot_{{\rm ev}_1} B,D \hot_{{\rm ev}_1} 1\big) \sim_u (X_1,D_1) ,
\]
where ${\rm ev}_t\colon C( [0,1],B) \to B$ denotes the even $*$-homomorphism given by evaluation at $t \in [0,1]$. In this case we write $(X_0,D_0) \sim_h (X_1,D_1)$.
\end{dfn}

Before proving that homotopies of unbounded Kasparov modules yields an equivalence relation it is worthwhile to spend a little time on a glueing construction for $\zz/2\zz$-graded $C^*$-correspondences. Consider two countably generated $\zz/2\zz$-graded $C^*$-correspondences~$X$ and~$X'$ both from $A$ to $C([0,1],B)$ and suppose that
\[
U\colon \ X \hot_{{\rm ev}_1} B \to X' \hot_{{\rm ev}_0} B
\]
is an even unitary isomorphism of $C^*$-correspondences. This data gives rise to a $\zz/2\zz$-graded $C^*$-correspondence $X \ti_U X'$ from $A$ to $C([0,1],B)$ obtained by glueing $X$ and $X'$ using the unitary $U$ to identify the fibres sitting at $1$ and $0$, respectively. Indeed, we may define
\[
X \ti_U X' := \big\{ (\xi,\xi') \in X \op X' \,|\, U\big( \xi \hot_{{\rm ev}_1} 1 \big) = \xi' \hot_{{\rm ev}_0} 1 \big\}
\]
and endow this set with the vector space structure, left action of $A$ and $\zz/2\zz$-grading inherited from the direct sum $X \op X'$. To construct the right action of $C([0,1],B)$ and the inner product on $X \ti_U X'$, we introduce the even $*$-endomorphisms $\be_0 , \be_1\colon C( [0,1],B) \to C([0,1],B)$ by $\be_0(f)(t) = f(t/2)$ and $\be_1(f)(t) = f( (t + 1)/2)$ for all $f \in C([0,1],B)$ and $t \in [0,1]$. We put
\begin{gather*}
\inn{ (\xi,\xi'), (\eta,\eta')}(t) :=
\begin{cases} \inn{\xi,\eta}(2t) & \text{for} \ t \in [0,1/2], \\
\inn{\xi',\eta'}(2t - 1) & \text{for} \ t \in [1/2,1],
\end{cases} \qquad \text{and} \\
(\xi,\xi') \cd f := (\xi \cd \be_0(f),\xi' \cd \be_1(f)) ,
\end{gather*}
for all $(\xi,\xi'), (\eta,\eta') \in X \ti_U X'$, $f \in C([0,1],B)$ and $t \in [0,1]$.

\begin{lem}\label{l:coucom}
The $\zz/2\zz$-graded $C^*$-correspondence $X \ti_U X'$ from $A$ to $C([0,1],B)$ is countably generated. Moreover, if $K\colon X \to X$ and $K'\colon X' \to X'$ are compact operators such that $U ( K \hot_{{\rm ev}_1} 1) U^* = K' \hot_{{\rm ev}_0} 1$, then the direct sum $K \op K'\colon X \op X' \to X \op X'$ restricts to a~compact operator $K \op K'\colon X \ti_U X' \to X \ti_U X'$.
\end{lem}
\begin{proof}
We start by constructing an adjointable isometry $S\colon X \ti_U X' \to \ell^2(\nn, C([0,1],B))$, where $\ell^2(\nn,C([0,1],B))$ denotes the standard module over $C([0,1],B)$. Since $B$ is $\si$-unital by our standing assumptions this will imply that $X \ti_U X'$ is countably generated.

Since $X$ and $X'$ are countably generated over $C([0,1],B)$ it follows by Kasparov's stabilization theorem, \cite[Theorem 2]{Kas:HSV}, that we may find unitary isomorphisms of Hilbert $C^*$-modules
\begin{gather*}
W\colon \ X \op \ell^2(\nn, C([0,1],B)) \to \ell^2(\nn,C([0,1],B)) \q \text{and} \\
W'\colon \ X' \op \ell^2(\nn, C([0,1],B)) \to \ell^2(\nn,C([0,1],B)) .
\end{gather*}
We let $V\colon \ell^2(\nn,B) \to \ell^2(\nn,B)$ denote the unique unitary isomorphism of Hilbert $C^*$-modules, which makes the diagram here below commute:
\[
\begin{CD}
\big( X \hot_{{\rm ev}_1} B \big) \op \ell^2(\nn,B) @>{ U \op 1}>> \big( X' \hot_{{\rm ev}_0} B \big) \op \ell^2(\nn,B) \\
@V{\cong}VV @VV{\cong}V \\
\big( X \op \ell^2(\nn, C( [0,1],B)) \big) \hot_{{\rm ev}_1} B & & \big( X' \op \ell^2(\nn, C( [0,1],B)) \big) \hot_{{\rm ev}_0} B \\
@V{W \hot_{{\rm ev}_1} 1}VV @VV{W' \hot_{{\rm ev}_0} 1}V \\
\ell^2(\nn,C( [0,1],B)) \hot_{{\rm ev}_1} B & & \ell^2(\nn,C( [0,1],B)) \hot_{{\rm ev}_0} B \\
@V{\cong}VV @VV{\cong}V \\
\ell^2(\nn,B) @>>{V}> \ell^2(\nn,B).
\end{CD}
\]
We specify that the lower vertical isomorphisms are induced by $f \hot_{{\rm ev}_i} b \mapsto f(i) \cd b$, for $i \in \{0,1\}$, and the top vertical unitary isomorphisms come from the distributivity of the interior tensor product together with the lower vertical isomorphisms.

The notation $\io\colon X \to X \op \ell^2(\nn,C([0,1],B))$ and $\io'\colon X' \to X' \op \ell^2(\nn,C([0,1],B))$ refers to the standard inclusions given on matrix form as $\left(\begin{smallmatrix} 1 \\ 0\end{smallmatrix}\right)$. We define our adjointable isometry $S\colon X \ti_U X' \to \ell^2( \nn, C([0,1],B))$ by the formula
\[
S(\xi,\xi')(t) := \begin{cases}
V( W \io \xi)(2t) & \text{for} \ t \in [0,1/2], \\ (W' \io' \xi')(2t - 1) & \text{for} \ t \in [1/2,1].\end{cases}
\]
for all $(\xi,\xi') \in X \ti_U X'$ and $t \in [0,1]$. We leave it to the reader to verify that $S$ is well-defined (in particular that $V(W \io \xi)(1) = (W' \io'\xi')(0)$). The adjoint of $S$ is given explicitly by
\[
S^*(f) = ( \io^* W^* \si_0(f), (\io')^* (W')^*\be_1(f)),
\]
where $\si_0(f)(t) = V^* f(t/2)$ and $\be_1(f) = f((t + 1)/2)$ for all $f \in \ell^2(C([0,1],B))$ and $t \in [0,1]$.

The compactness of $K \op K'\colon X \ti_U X' \to X \ti_U X'$ is equivalent to the compactness of $S (K \op K') S^*\colon \ell^2(\nn,C([0,1],B)) \to \ell^2(\nn,C([0,1],B))$. The compact operators on the standard module $\ell^2(\nn,C([0,1],B))$ can be identified with the operator norm continuous maps from $[0,1]$ to the compact operators on $\ell^2(\nn,B)$. Using this identification, we compute that
\[
S (K \op K') S^* (t) = \begin{cases} V L(2t) V^* & \text{for} \ t \in [0,1/2], \\ L'(2t - 1) & \text{for} \ t \in [1/2,1] ,\end{cases}
\]
where $L := W \io K \io^* W^*$ and $L' := W' \io' K' (\io')^* (W')^*\colon \ell^2(\nn,C([0,1],B)) \to \ell^2(\nn,C([0,1],B))$. Since $L$ and $L'$ are compact by assumption this proves the compactness of $K \op K'\colon X \ti_U X' \to X \ti_U X'$.
\end{proof}

\begin{prop}\label{p:unbhom}
Homotopy of unbounded Kasparov modules is an equivalence relation.
\end{prop}
\begin{proof}{\it Reflexivity:} For an unbounded Kasparov module $(X,D)$ from $\s A$ to $B$, we have that $(X,D) \sim_h (X,D)$ via the unbounded Kasparov module $( C([0,1],X), E)$ from $\s A$ to $C([0,1],B)$, where $E\colon \operatorname{Dom}(E) \to C( [0,1],X)$ is the unbounded selfadjoint and regular operator defined by
\begin{gather*}
\operatorname{Dom}(E) := \{ \xi \in C( [0,1],X) \,|\, \xi(t) \in \operatorname{Dom}(D) \ \text{for all} \ t \in [0,1] \\
\hphantom{\operatorname{Dom}(E) := \{}{} \text{and} \ t \mapsto D \xi(t) \ \text{is continuous} \}, \\
E(\xi)(t) := D \xi(t) \qquad \text{for all } t \in [0,1] .
\end{gather*}

{\it Symmetry:} Suppose that two unbounded Kasparov modules $(X_0,D_0)$ and $(X_1,D_1)$ from~$\s A$ to~$B$ are homotopic via the unbounded Kasparov module $(X,D)$ from $\s A$ to $C( [0,1],B)$. Define the even $*$-automorphism $\be\colon C( [0,1],B) \to C( [0,1] , B)$ by $\be(f)(t) = f(1 - t)$ for all $f \in C( [0,1],B)$ and $t \in [0,1]$. Then it holds that $(X_1,D_1)$ and $(X_0,D_0)$ are homotopic via the unbounded Kasparov module $(X \hot_{\be} C( [0,1],B), D \hot_{\be} 1)$ from $\s A$ to $C([0,1],B)$.

{\it Transitivity:} Suppose that $(X_0,D_0)$, $(X_1,D_1)$ and $(X_1,D_1')$ are unbounded Kasparov modules from $\s A$ to $B$ such that $(X_0,D_0) \sim_h (X_1,D_1)$ and $(X_1,D_1) \sim_h (X_1',D_1')$ via the unbounded Kasparov modules $(X,D)$ and $(X', D')$, respectively. Let us choose an even unitary isomorphism of $C^*$-correspondences
\[
U\colon \ X \hot_{{\rm ev}_1} B \to X' \hot_{{\rm ev}_0} B
\]
such that $U (D \hot_{{\rm ev}_1} 1) U^* = D' \hot_{{\rm ev}_0} 1$. We study the associated $\zz/2\zz$-graded $C^*$-correspondence $X \ti_U X'$ from $A$ to $C([0,1],B)$ and notice that $X \ti_U X'$ is countably generated by Lemma~\ref{l:coucom}. We define the odd unbounded selfadjoint and regular operator $D''\colon \operatorname{Dom}(D'') \to X \ti_U X'$ by
\begin{gather*}
\operatorname{Dom}(D'') := \{ (\xi,\xi') \in X \ti_U X' \,|\, \xi \in \operatorname{Dom}(D) , \, \xi' \in \operatorname{Dom}(D') \}, \\
D''(\xi,\xi') := (D \xi, D' \xi') .
\end{gather*}
To see that $D''$ is indeed selfadjoint and regular, we notice that~$D''$ is symmetric and that the resolvents are given by
\[
(D'' + \la i )^{-1} (\xi,\xi') = \big( (D + \la i )^{-1}(\xi), (D' + \la i )^{-1}(\xi')\big)
\]
for all $\la \in \rr \setminus \{0\}$ and all $(\xi,\xi') \in X \ti_U X'$.

We claim that $(X \ti_U X',D'')$ is an unbounded Kasparov module from $\s A$ to $C([0,1],B)$. It is indeed clear that each $a \in \s A$ preserves the domain of $D''$ and that the graded commutator $[ D'', a ]\colon \operatorname{Dom}(D'') \to X \ti_U X'$ extends to a bounded adjointable operator on $X \ti_U X'$. Moreover, it follows from Lemma~\ref{l:coucom} that $a \cd (i + D'')^{-1}\colon X \ti_U X' \to X \ti_U X'$ is a compact operator for all $a \in A$.

The unbounded Kasparov module $(X \ti_U X',D'')$ from $\s A$ to $C( [0,1],B)$ implements the homotopy from $(X_0, D_0)$ to $(X_1',D_1')$ and this ends the proof of the proposition.
\end{proof}

Let us shortly discuss the notion of bounded perturbations of unbounded Kasparov modules:

\begin{dfnprop}\label{p:bddper}
Let $(X,D)$ be an unbounded Kasparov module from $\s A$ to~$B$ and let $T\colon X \to X$ be an odd bounded selfadjoint operator. Then the pair $(X, D + T)$ is an unbounded Kasparov module from $\s A$ to $B$ and $(X,D)$ and $(X, D + T)$ are homotopic. We say that $(X, D + T)$ is a \emph{bounded perturbation} of $(X,D)$.
\end{dfnprop}
\begin{proof}
The domain of $D + T$ agrees with the domain of $D$ and $D + T\colon \operatorname{Dom}(D) \to X$ is selfadjoint and regular by \cite[Example 1]{Wor:UAQ}. The commutator condition in Definition \ref{d:unbkas} is immediately verified and the resolvent condition in Definition \ref{d:unbkas} follows from the resolvent identity:
\[
a \cd (i + D + T)^{-1} = a \cd (i + D)^{-1} - a \cd (i + D)^{-1} T (i + D + T)^{-1} .
\]
We thus have that $(X, D + T)$ is an unbounded Kasparov module from $\s A$ to $B$. To see that $(X,D)$ and $(X, D + T)$ are homotopic we apply the unbounded Kasparov module from $\s A$ to $C( [0,1],B)$ given by $( C( [0,1],X), E)$ where the corresponding odd unbounded selfadjoint and regular operator is defined by
\begin{gather*}
\operatorname{Dom}(E) := \{ \xi \in C( [0,1],X) \,|\, \xi(t) \in \operatorname{Dom}(D) \ \text{for all} \ t \in [0,1] \\
\hphantom{\operatorname{Dom}(E) := \{}{} \text{and} \ t \mapsto D \xi(t) \ \text{is continuous} \},\\
E(\xi)(t) := (D + t \cd T)\xi(t) \qquad \text{for all } t \in [0,1] . \tag*{\qed}
\end{gather*}\renewcommand{\qed}{}
\end{proof}

We remark that the above notion of bounded perturbations yields an equivalence relation on unbounded Kasparov modules. We will not discuss this equivalence relation any further at this point but refer the reader to~\cite{Kaa:MEU} for more details.

In this paper the relevant equivalence relation on unbounded Kasparov modules is a stabilized version of homotopies of unbounded Kasparov modules, where we are using spectrally decomposable unbounded Kasparov modules in the stabilization procedure:

\begin{dfn}
Two unbounded Kasparov modules $(X_0,D_0)$ and $(X_1,D_1)$ from $\s A$ to $B$ are \emph{stably homotopic} when there exist two spectrally decomposable unbounded Kasparov modules $(Y_0,D_0')$ and $(Y_1,D_1')$ from $\s A$ to $B$ such that
\[
(X_0 \op Y_0,D_0 \op D_0') \sim_h (X_1 \op Y_1, D_1 \op D_1') .
\]
In this case we write $(X_0,D_0) \sim_{sh} (X_1,D_1)$.
\end{dfn}

We remark that the relation $\sim_{sh}$ is indeed an equivalence relation and that this can be verified using Proposition \ref{p:unbhom} together with the fact that the direct sum of two spectrally decomposable unbounded Kasparov modules is again spectrally decomposable.

\begin{dfn}
The \emph{topological unbounded $KK$-theory} from $\s A$ to $B$ consists of the unbounded Kasparov modules from $\s A$ to $B$ modulo stable homotopies, thus modulo the equivalence relation~$\sim_{sh}$. The topological unbounded $KK$-theory from $\s A$ to $B$ is denoted by $UK^{{\rm top}}(\s A, B)$.
\end{dfn}

We may equip the topological unbounded $KK$-theory from $\s A$ to $B$ with the structure of a~commutative monoid, where the addition is induced by the direct sum operation from Definition~\ref{d:dirsum} and the neutral element is the class of the zero module $(0,0)$. We shall see in Section~\ref{s:group} that $UK^{{\rm top}}(\s A,B)$ is in fact an abelian group.

The next result is a combination of Theorem \ref{t:bousur} and Lemma \ref{l:spezer} together with the observation that the Baaj--Julg bounded transform is compatible with direct sums and homotopies of unbounded Kasparov modules.

\begin{Theorem}\label{t:baajul}
The Baaj--Julg bounded transform $(X,D) \mapsto \big[ X, D\big(1 + D^2\big)^{-1/2}\big]$ induces a~well-defined surjective homomorphism
\[
\C F\colon \ UK^{{\rm top}}(\s A, B) \to KK(A,B) ,
\]
which we also refer to as the Baaj--Julg bounded transform.
\end{Theorem}

The main result of this paper is that the surjective homomorphism $\C F\colon UK^{{\rm top}}(\s A,B) \to KK(A,B)$ is in fact an isomorphism. In particular, it holds that $UK^{{\rm top}}(\s A,B)$ is independent of the norm-dense $\zz/2\zz$-graded $*$-subalgebra $\s A \su A$ as long as $\s A$ is countably generated. This will be proved in Section \ref{s:injbou}.

\section{Lipschitz regularity and invertibility}
We shall now see that given a class in topological unbounded $KK$-theory one may always choose a~Lipschitz regular representative $(X,D)$ with the extra property that the unbounded selfadjoint and regular operator $D\colon \operatorname{Dom}(D) \to X$ is invertible (so that $D^{-1}\colon X \to X$ is a~bounded adjointable operator with image equal to the domain of $D$).

We start with Lipschitz regularity:

\begin{prop}\label{p:liphom}Let $r \in (0,1/2)$ and suppose that $(X,D)$ is an unbounded Kasparov module from~$\s A$ to~$B$. Then $\big(X,D\big(1 + D^2\big)^{-r}\big)$ is a Lipschitz regular unbounded Kasparov module from~$\s A$ to~$B$ and it holds that
\[
\big(X, D\big(1 + D^2\big)^{-r}\big) \sim_h (X,D) .
\]
\end{prop}
\begin{proof}For each $t \in [0,1]$, define the functions
\[
f_t^{\pm}\colon \ \B R \to \B C, \qquad f_t^{\pm}(x) := \big( i \pm x \big(1 + x^2\big)^{-tr} \big)^{-1} .
\]
We remark that $f_t^{\pm} \in C_0( \B R )$ and that the maps $[0,1] \to C_0(\B R)$ given by $t \mapsto f_t^{\pm}$ are continuous with respect to the supremum norm on $C_0( \B R)$. Notice in this respect that we have the estimates $| f_t^{\pm}(x) | \leq 2^r |x|^{2r - 1}$ whenever $|x| \geq 1$ and $t \in [0,1]$.

For each $a \in A$, we thus have that
\[
a \cd \big(i \pm D\big(1 + D^2\big)^{-tr}\big)^{-1} = a \cd f_t^{\pm}(D)\colon \ X \to X
\]
are compact operators for all $t \in [0,1]$ and that the maps $[0,1] \to \B L(X)$ given by $t \mapsto \big(i \pm D\big(1 + D^2\big)^{-tr}\big)^{-1}$ are continuous in operator norm.

For $p \in (0,1/2]$, we are going to apply the integral formula
\begin{gather}\label{eq:intpow}
\big(1 + D^2\big)^{-p} = \frac{\sin(p \pi)}{\pi} \int_0^\infty \la^{-p} \big(1 + \la + D^2\big)^{-1} d\la ,
\end{gather}
where the integrand is continuous in operator norm and the integral converges absolutely in operator norm.

Let $a \in \s A$ be homogeneous. For each $t \in [0,1]$, the domain of $D$ is a core for $D\big(1 + D^2\big)^{-tr}$ and for each $\xi \in \operatorname{Dom}(D)$, we compute the graded commutator
\begin{gather}\label{eq:comsum}
\big[ D \big(1 + D^2\big)^{- tr} , a \big](\xi) = \big(1 + D^2\big)^{-tr} [D,a] (\xi) + (-1)^{\deg (a)}\big[ \big(1 + D^2\big)^{-tr},a\big] D(\xi) .
\end{gather}
The first term extends to the bounded adjointable operator $\big(1 + D^2\big)^{-tr} d(a)\colon X \to X$ and we remark that the map $[0,1] \to \B L(X)$ defined by $t \mapsto \big(1 + D^2\big)^{-tr} d(a)$ is continuous with respect to the strict operator topology on $\B L(X)$. The second term in equation~\eqref{eq:comsum} is more complicated and, using the integral formula in equation~\eqref{eq:intpow}, we obtain the expression
\begin{gather}
\big[ \big(1 + D^2\big)^{-tr},a\big] D \nonumber\\
\qquad{} = (-1)^{\deg (a) + 1} \frac{\sin(tr \pi)}{\pi} \int_0^\infty \la^{-tr} \big(1 + \la + D^2\big)^{-1} d(a) D^2 \big(1 + \la + D^2\big)^{-1} d\la \nonumber\\
\qquad\quad{} - \frac{\sin(tr \pi)}{\pi} \int_0^\infty \la^{-tr} D \big(1 + \la + D^2\big)^{-1} d(a) D \big(1 + \la + D^2\big)^{-1} d\la\label{eq:intder}
\end{gather}
for all $t \in (0,1]$, where the left hand side only makes sense on $\operatorname{Dom}(D)$. The right hand side does however make sense as a bounded adjointable operator on $X$ and the operator norm is dominated by
\[
\frac{2 \sin(tr \pi) \| d(a) \| }{\pi} \int_0^\infty \la^{-tr} (1 + \la)^{-1} d\la = 2 \| d(a) \| .
\]
For each $t \in [0,1]$ we denote the bounded adjointable extension of $\big[\big(1 + D^2\big)^{-tr},a\big] D\colon \operatorname{Dom}(D) \allowbreak \to X$ by $G_t(a)\colon X \to X$. Notice that $G_0(a) = 0$ and that $G_t(a)$ is given explicitly by the right hand side of equation \eqref{eq:intder} for $t \in (0,1]$. In particular, it holds that $\| G_t(a) \| \leq 2 \| d(a) \|$ for all $t \in [0,1]$. Using the identity in equation~\eqref{eq:comsum} one may also verify that
\begin{equation}\label{eq:adjstr}
G_t(a)^* + (-1)^{\deg (a)}G_t(a^*) = \big[ d(a^*), \big(1 + D^2\big)^{-tr} \big]
\end{equation}
for all $t \in [0,1]$. We claim that the map $[0,1] \to \B L(X)$ given by $t \mapsto G_t(a)$ is strictly continuous. Since $\s A$ is a $*$-algebra and since the right hand side of the identity in equation~\eqref{eq:adjstr} defines a strictly continuous map on $[0,1]$ we only need to show that the map $t \mapsto G_t(a) \xi$ is norm continuous for every $\xi \in X$. In fact, because of the uniform bound on the operator norm of~$G_t(a)$ for $t \in [0,1]$ and the density of $\operatorname{Dom}(D)$ in~$X$, we may restrict our attention to elements $\xi \in \operatorname{Dom}(D)$. But for $\xi \in \operatorname{Dom}(D)$ the norm continuity of the map $t \mapsto G_t(a) \xi$ follows since $t \mapsto \big[ \big(1 + D^2\big)^{-tr}, a\big]$ is strictly continuous and since $G_t(a) \xi = \big[ \big(1 + D^2\big)^{-tr},a\big] D \xi$ for all $t \in [0,1]$.

Our efforts so far can be summarized as follows: we have an unbounded Kasparov module $( C([0,1],X), E)$ from $\s A$ to $C([0,1],B)$, where the unbounded selfadjoint and regular operator $E\colon \operatorname{Dom}(E) \to C( [0,1],X)$ is defined by
\begin{gather*}
\operatorname{Dom}(E) := \big\{ \xi \in C( [0,1],X) \,|\, \xi(t) \in \operatorname{Dom}\big( D(1 + D^2)^{-tr}\big) \ \text{for all} \ t \in [0,1] \\
\hphantom{\operatorname{Dom}(E) :=\big\{}{} \text{and} \ t \mapsto D\big(1 + D^2\big)^{-tr} \xi(t) \ \text{is continuous} \big\}, \\
E(\xi)(t) := D\big(1 + D^2\big)^{-tr} \xi(t) \q \text{for all } t \in [0,1] .
\end{gather*}
In particular, we have that the unbounded Kasparov modules $(X,D)$ and $\big(X, D\big(1 + D^2\big)^{-r}\big)$ are homotopic.

To finish the proof of the proposition we only need to argue that the unbounded Kasparov module $\big(X,D\big(1 + D^2\big)^{-r}\big)$ is Lipschitz regular. Let $a \in \s A$ be homogeneous. Since the function $x \mapsto |x| - x^2 \big(1 + x^2\big)^{-1/2}$ is bounded on $\B R$, we just have to prove that the graded commutator
\[
\big[ D\big(1 + D^2\big)^{-1/2} D\big(1 + D^2\big)^{-r}, a \big]\colon \ \operatorname{Dom}(D) \to X
\]
extends to a bounded operator on~$X$. Notice in this respect that
\[ D\big(1 + D^2\big)^{-1/2} D\big(1 + D^2\big)^{-r}\colon \ \operatorname{Dom}\big(|D|^{1 - 2r}\big) \to X\] agrees with
\[ \big|D\big(1 + D^2\big)^{-r}\big| = |D|\big(1 + D^2\big)^{-r}\colon \ \operatorname{Dom}\big(|D|^{1-2r}\big) \to X\] up to a bounded selfadjoint operator and moreover that $\operatorname{Dom}(D) \su X$ is a core for $\big|D \big(1 + D^2\big)^{-r}\big|$. Since we already know that
\[
[D,a] D \big(1 + D^2\big)^{-r - 1/2} \ \text{and} \ D\big(1 + D^2\big)^{-1/2} \big[D\big(1 + D^2\big)^{-r}, a\big]\colon \ \operatorname{Dom}(D) \to X
\]
extend to bounded operators on $X$ we are left with proving that
\[
D \big[ \big(1 + D^2\big)^{-1/2}, a\big] D\big(1 + D^2\big)^{-r}\colon \ \operatorname{Dom}(D) \to X
\]
extends to a bounded operator on~$X$. But this follows since the integral formula in equation~\eqref{eq:intpow} implies that
\begin{gather}
 D \big[ \big(1 + D^2\big)^{-1/2}, a\big] D\big(1 + D^2\big)^{-r} \nonumber\\
\qquad{} = (-1)^{\deg (a) + 1} \frac{1}{\pi} \int_0^\infty \la^{-1/2} D \big(1 + \la + D^2\big)^{-1} d(a) D^2 \big(1 + \la + D^2\big)^{-1} \big(1 + D^2\big)^{-r} d\la \nonumber\\
\qquad\quad{}-\frac{1}{\pi} \int_0^\infty \la^{-1/2} D^2 \big(1 + \la + D^2\big)^{-1} d(a) D\big(1 + \la + D^2\big)^{-1}\big(1 + D^2\big)^{-r} d\la ,\label{eq:lipcom}
\end{gather}
where the left hand side only makes sense on $\operatorname{Dom}\big(|D|^{1-2r}\big)$, but the right hand side makes sense as a bounded operator on $X$. Indeed, both of the integrals in equation~\eqref{eq:lipcom} have operator norm continuous integrands and converge absolutely because of the operator norm estimates
\begin{gather*}
 \big\| \la^{-1/2} D \big(1 + \la + D^2\big)^{-1} d(a) D^2 \big(1 + \la + D^2\big)^{-1} \big(1 + D^2\big)^{-r} \big\| , \\
 \big\| \la^{-1/2} D^2 \big(1 + \la + D^2\big)^{-1} d(a) D\big(1 + \la + D^2\big)^{-1}\big(1 + D^2\big)^{-r} \big\| \\
 \qquad{} \leq \| d(a) \| \cd \la^{-1/2} (1 + \la)^{-1/2 - r} ,
\end{gather*}
which are valid for all $\la \in (0,\infty)$.
\end{proof}

We continue with invertibility:

\begin{prop}\label{p:lipinv}
Suppose that $(X,D)$ is an unbounded Kasparov module from~$\s A$ to~$B$. Then $(X,D)$ is homotopic to a Lipschitz regular unbounded Kasparov module $(X', D')$ with $D'\colon$ $\operatorname{Dom}(D') \to X'$ invertible.
\end{prop}
\begin{proof}By Proposition \ref{p:liphom} we may assume that $(X,D)$ is already Lipschitz regular. Let us denote the $\zz/2\zz$-grading operator on $X$ by $\ga\colon X \to X$. Define the $\zz/2\zz$-graded $C^*$-corresponden\-ce~$\wit{X}$ from $A$ to $B$ which agrees with~$X$ as a~Hilbert $C^*$-module over~$B$, but $\wit{X}$ has grading operator $-\ga\colon \wit{X} \to \wit{X}$ and the left action of $A$ is trivial. Then the unbounded Kasparov module $\big(\wit{X},-D\big)$ is homotopic to the zero module $(0,0)$. Indeed, we may consider the $\zz/2\zz$-graded $C^*$-correspondence $C_0\big( (0,1],\wit{X}\big)$ from $A$ to $C( [0,1],B)$ equipped with the odd unbounded selfadjoint and regular operator $E\colon \operatorname{Dom}(E) \to C_0\big( (0,1],\wit{X}\big)$ defined by
\begin{gather*}
\begin{split}
& \operatorname{Dom}(E) := \big\{ \xi \in C_0\big( (0,1],\wit{X}\big) \,|\, \xi(t) \in \operatorname{Dom}(D) \ \text{for all} \ t \in (0,1] \\
& \hphantom{\operatorname{Dom}(E) := \big\{}{} \text{and} \ t \mapsto - D \xi(t) \ \text{is continuous and vanishes at zero} \big\},\\
& E(\xi)(t) := - D \xi(t) \qquad \text{for all } t \in (0,1] .
\end{split}
\end{gather*}
Since the left action of $A$ on $C_0\big( (0,1],\wit{X}\big)$ is trivial we have that $\big( C_0\big( (0,1],\wit{X}\big), E\big)$ is an unbounded Kasparov module from $\s A$ to $C([0,1],B)$ thus realizing the homotopy from $\big(\wit{X},-D\big)$ to $(0,0)$. The result of the proposition now follows from Proposition~\ref{p:bddper} by noting that $(X,D) + \big(\wit{X}, -D\big) = \big(X \op \wit{X}, D \op (-D)\big)$ is a bounded perturbation of the Lipschitz regular unbounded Kasparov module
\[
\left(X \op \wit{X}, \left(\begin{matrix} D & 1 \\ 1 & - D\end{matrix}\right)\right)
\]
from $\s A$ to $B$. Remark in this respect that the square of $\left(\begin{smallmatrix}D & 1 \\ 1 & - D\end{smallmatrix}\right)\colon \operatorname{Dom}(D) \op \operatorname{Dom}(D) \to X \op \wit{X}$ is given by
\[
\left(\begin{matrix} D & 1 \\ 1 & - D\end{matrix}\right)^2 = \left(\begin{matrix} 1 + D^2 & 0 \\ 0 & 1 + D^2 \end{matrix}\right)\colon \ \operatorname{Dom}\big(D^2\big) \op \operatorname{Dom}\big(D^2\big) \to X \op \wit{X} ,
\]
which is indeed an invertible operator.
\end{proof}

\section{Group structure}\label{s:group}
We show in this section that the commutative monoid $UK^{{\rm top}}(\sA,B)$ is in fact an abelian group. This result relies on a more general proposition stating (at least roughly speaking) that two unbounded Kasparov modules $(X,D)$ and $(X,D')$ from $\s A$ to $B$ are stably homotopic when the odd unbounded selfadjoint and regular operators $D$ and $D'$ have the same phase. This proposition will also be of key importance later on when we prove the injectivity of the Baaj--Julg bounded transform.

\begin{dfn}\label{d:inv}The \emph{inverse} of an unbounded Kasparov modules $(X,D)$ from $\s A$ to $B$ is the unbounded Kasparov module from $\s A$ to $B$ given by
\[
-(X,D) := \big( X^{{\rm op}},-D\big) ,
\]
where $X^{{\rm op}}$ agrees with $X$ as a Hilbert $C^*$-module over $B$, but $X^{{\rm op}}$ is equipped with the opposite $\zz/2\zz$-grading and with left action $\pi_{X^{{\rm op}}}\colon A \to \B L\big(X^{{\rm op}}\big)$ defined by
\[
\pi_{X^{{\rm op}}}(a) = \begin{cases} \pi_X(a) & \mbox{for} \ a \mbox{ even}, \\ -\pi_X(a) & \mbox{for} \ a \mbox{ odd}.\end{cases}
\]
\end{dfn}

\begin{prop}\label{p:phaequ}
Let $(X,D)$ and $(X,D')$ be two unbounded Kasparov modules from $\s A$ to $B$ and suppose there exists an odd selfadjoint unitary operator $F\colon X \to X$ such that
\begin{enumerate}\itemsep=0pt
\item[$1)$] the operator $F$ preserves the domain of $D$ and of $D'$ and the commutators $F D - D F\colon$ $\operatorname{Dom}(D) \to X$ and $F D' - D' F\colon \operatorname{Dom}(D') \to X$ have bounded extensions to~$X$;
\item[$2)$] the unbounded operators $D F$ and $D' F$ are bounded perturbations of even unbounded positive and regular operators $\De\colon \operatorname{Dom}(DF) \to X$ and $\De'\colon \operatorname{Dom}(D'F) \to X$;
\item[$3)$] for each $a \in \s A$, the image of the graded commutator $[F,a]$ is contained in $\operatorname{Dom}(D) \cap \operatorname{Dom}(D')$ and the operators $D [F,a], D' [F,a]\colon X \to X$ are bounded.
\end{enumerate}
Then $(X,D) - (X,D')$ is stably homotopic to the zero module $(0,0)$.
\end{prop}
\begin{proof}
We are going to show that
\[
(X,D) - (X, D') = \big(X \op X^{{\rm op}}, D \op (- D') \big)
\]
is homotopic to a spectrally decomposable unbounded Kasparov module. We denote the grading operator on $X$ by $\ga\colon X \to X$ so that the grading operator on $X^{{\rm op}}$ is given by $-\ga\colon X \to X$.

We remark that the graded commutator $[F,a]\colon X \to X$ is compact for all $a \in A$. Indeed, for any $a \in \s A$ and $a' \in A$, it follows from assumption $(3)$ and the fact that $(X,D)$ is an unbounded Kasparov module that
\[
a' [F,a] = a' (i + D)^{-1} (i + D) [F,a]\colon \ X \to X
\]
is a compact operator.

Define the orthogonal projections
\[
P_+ := \frac{1 + F}{2} , \ P_- := \frac{1 - F}{2} = 1 - P_+\colon \ X \to X
\]
and for each $t \in [0,1]$, define the unitary automorphisms of the Hilbert $C^*$-module $X \op X$:
\begin{gather}
U_t := \left(\begin{matrix} \cos(t \pi/ 2) \cd P_- + P_+ & \sin(t \pi/2) \cd P_- \\ - \sin(t \pi/2) \cd P_- & \cos(t \pi/2) \cd P_- + P_+\end{matrix}\right) \qquad \text{and} \nonumber\\
V_t := \left(\begin{matrix} \cos(t \pi/ 2) & \sin(t \pi/2) \\ - \sin(t \pi/2) & \cos(t \pi/2) \end{matrix}\right) .\label{eq:unitary}
\end{gather}

We study the $\zz/2\zz$-graded $C^*$-correspondence $Y$ from $A$ to $C([0,1],B)$, which as a Hilbert $C^*$-module over $C([0,1],B)$ agrees with $C( [0,1],X \op X)$, but with left action $\pi\colon A \to \B L(Y)$ defined by putting
\[
\pi_t(a) := \begin{cases} U_t \left(\begin{matrix}\pi_X(a) & 0 \\ 0 & \pi_X(a)\end{matrix}\right) U_t^* & \text{for} \ a \text{ even}, \vspace{1mm}\\
U_t \left(\begin{matrix}0 & -\pi_X(a) \\ -\pi_X(a) & 0\end{matrix}\right) V_t U_t^*
& \text{for} \ a \text{ odd} ,\end{cases}
\]
for all $t \in [0,1]$, and with grading operator $\si\colon Y \to Y$ defined by
\[
\si_t := \left(\begin{matrix} 0 & - \ga \\ - \ga & 0\end{matrix}\right) \colon \ X \op X \to X \op X \qquad \text{for all } t \in [0,1] .
\]
It is useful to notice that
\[
\si_t = U_t \left(\begin{matrix} 0 & - \ga \\ -\ga & 0\end{matrix}\right) V_t U_t^* \qquad \text{for all } t \in [0,1] .
\]

Using assumption $(1)$, we define the unbounded selfadjoint and regular operators
\begin{gather*}
D_+ := P_+ D P_+\colon \ \operatorname{Dom}(D P_+) \to X , \qquad D_- := P_- D P_-\colon \ \operatorname{Dom}(D P_-) \to X \qquad \text{and} \\
D'_+ := P_+ D' P_+\colon \ \operatorname{Dom}(D' P_+) \to X , \qquad D'_- := P_- D' P_-\colon \ \operatorname{Dom}(D P_-) \to X .
\end{gather*}
One may then verify directly that
\begin{gather*}
 D_+ - D_-'\colon \ \operatorname{Dom}(D P_+) \cap \operatorname{Dom}(D' P_-) \to X \qquad \text{and} \\
 D_- - D_+'\colon \ \operatorname{Dom}(D P_-) \cap \operatorname{Dom}(D' P_+) \to X
\end{gather*}
are unbounded selfadjoint and regular operators. The resolvent of $D_+ - D_-'$ is for example given by
\[
(i \la + D_+ - D_-')^{-1} = P_+ ( i \la + D_+)^{-1} + P_- (i\la - D_-')^{-1}
\]
for all $\la \in \B R \sem \{0\}$. Alternatively, we refer to \cite[Section~7]{KaLe:LGR} or \cite[Theorem~4.5]{LeMe:SSH} for much more general results on sums of unbounded selfadjoint and regular operators.

We claim that the pair $\left(Y, \left(\begin{smallmatrix} D_+ - D'_- & 0 \\ 0 & D_- - D_+' \end{smallmatrix}\right)\right)$ is an unbounded Kasparov module from $\s A$ to $C( [0,1], B)$ (where it is understood that the unbounded selfadjoint and regular operator in question acts as $( D_+ - D_-') \op (D_- - D_+')$ in each fibre). For each $t \in [0,1]$ we thus have the fibre $(Y_t, ( D_+ - D_-') \op (D_- - D_+'))$ where $Y_t$ is the countably generated $\zz/2\zz$-graded $C^*$-correspondence from $A$ to $B$ which agrees with $X \op X$ as a Hilbert $C^*$-module but with grading operator $\si_t = \left(\begin{smallmatrix}0 & -\ga \\ -\ga & 0\end{smallmatrix}\right)$ and with left action given by the even $*$-homomorphism $\pi_t\colon A \to \B L(X \op X)$.

We let $t \in [0,1]$ be given and compute for each even $a \in A$ that
\begin{gather*}
\pi_t(a) - \pi_0(a)
 = \frac{1}{2} \left(\begin{matrix} [F,a] & 0 \\ 0 & [F,a]\end{matrix}\right) \left(\begin{matrix} 1 - \cos(t \pi/2) & - \sin(t \pi /2) \\ \sin(t \pi/2) & 1 - \cos( t \pi/2)\end{matrix}\right) U_t^* \nonumber\\
\hphantom{\pi_t(a) - \pi_0(a)=}{} = -\frac{1}{2} U_t \left(\begin{matrix} 1 - \cos(t \pi/2) & \sin(t \pi /2) \\ -\sin(t \pi/2) & 1 - \cos( t \pi/2)\end{matrix}\right) \left(\begin{matrix} [F,a] & 0 \\ 0 & [F,a]\end{matrix}\right)
\end{gather*}
and for each odd $a \in A$ that
\begin{gather*}
\pi_t(a) - \pi_0(a)
 = - \frac{1}{2} \left(\begin{matrix} [F,a] & 0 \\ 0 & [F,a]\end{matrix}\right)
\left(\begin{matrix} -\sin(t \pi / 2) & 1 - \cos(t \pi/2) \\ 1 - \cos( t \pi/2) & \sin(t \pi/2) \end{matrix}\right) V_t U_t^* \\
\hphantom{\pi_t(a) - \pi_0(a)=}{} = -\frac{1}{2} U_t V_t^* \left(\begin{matrix} -\sin(t \pi / 2) & 1 - \cos(t \pi/2) \\ 1 - \cos( t \pi/2) & \sin(t \pi/2) \end{matrix}\right) \left(\begin{matrix} [F,a] & 0 \\ 0 & [F,a]\end{matrix}\right) .
\end{gather*}
Since the graded commutator $[F,a]\colon X \to X$ is compact, this computation implies that $\pi_t(a) - \pi_0(a)\colon X \op X \to X \op X$ is a compact operator for all $a \in A$ and $t \in [0,1]$ and moreover that the associated map $[0,1] \to \B K(X \op X)$ is continuous in operator norm. Using assumption $(3)$, the above computation also implies that $\pi_t(a) - \pi_0(a)$ preserves the domain of $(D_+ - D_-') \op (D_- - D_+')$ for all $a \in \s A$ and $t \in [0,1]$ (the image of $\pi_t(a) - \pi_0(a)$ is in fact contained in this domain) and moreover that the graded commutator
\begin{gather*}
 [ (D_+ - D_-') \op (D_- - D_+') , \pi_t(a) - \pi_0(a) ]
 \colon \\
 \qquad{} \operatorname{Dom}(DP_+ \cap D'P_-) \op \operatorname{Dom}(D P_- \op D'P_+) \to X \op X
\end{gather*}
extends to a bounded operator on $X \op X$ (in fact each of the terms have this property). The associated map $[0,1] \to \B L(X \op X)$ given by
\[
t \mapsto \ov{ [ (D_+ - D_-') \op (D_- - D_+') , \pi_t(a) - \pi_0(a) ] }
\]
is continuous in operator norm. These observations imply that $(Y, (D_+ - D_-') \op (D_- - D_+'))$ is an unbounded Kasparov module from $\s A$ to $C( [0,1],B)$ if and only if the fibre at $t = 1$ is an unbounded Kasparov module from $\s A$ to $B$.

The fibre at $t = 1$ is given by the pair $( Y_1, (D_+ - D_-') \op (D_- - D_+') )$, where $Y_1$ is unitarily isomorphic to $X \op X^{{\rm op}}$ as a $\zz/2\zz$-graded $C^*$-correspondence via the unitary operator $U_1\colon X \op X \to X \op X$ defined in equation~ \eqref{eq:unitary}. Moreover, we have that
\begin{gather*}
 U_1 ( D \op (- D') ) U_1^*
= \left(\begin{matrix} D_+ - D'_- & - P_+ D P_- - P_- D' P_+ \\ - P_- D P_+ - P_+ D' P_- & D_- - D'_+\end{matrix}\right)\colon \nonumber\\
\qquad{} \big( \operatorname{Dom}( DP_+) \cap \operatorname{Dom}( D' P_-) \big) \op \big( \operatorname{Dom}( D'P_+) \cap \operatorname{Dom}( D P_-) \big)
\to X \op X .
\end{gather*}
By assumption $(1)$ we know that the off-diagonal entries extend to bounded operators on $X$ and it follows that the fibre at $t = 1$ agrees with the unbounded Kasparov module $(X \op X^{{\rm op}}, D \op (- D') )$ from $\s A$ to $B$ up to unitary equivalence and bounded perturbations. By Proposition \ref{p:bddper} this implies in particular that $(Y_1, (D_+ - D_-') \op (D_- - D_+'))$ is an unbounded Kasparov module from~$\s A$ to~$B$ which is homotopic to $(X \op X^{{\rm op}}, D \op (- D')) = (X,D) - (X, D')$.

We may thus conclude that $(Y,(D_+ - D_-') \op (D_- - D_+'))$ is an unbounded Kasparov module from $\s A$ to $C( [0,1],B)$.

By what has been achieved so far, we have reduced the proof of the proposition to showing that the fibre of $(Y, (D_+ - D_-') \op (D_- - D_+'))$ at $t = 0$ is homotopic to a spectrally decomposable unbounded Kasparov module from $\s A$ to $B$.

The unbounded Kasparov module sitting as the fibre at $t = 0$ is given by the pair $( Y_0, (D_+ - D_-') \op (D_- - D_+') )$, where $Y_0$ agrees with $X \op X$ as a Hilbert $C^*$-module over $B$ but with grading operator $\si_0 = \left(\begin{smallmatrix} 0 & - \ga \\ - \ga & 0\end{smallmatrix}\right)$ and left action $\pi_0\colon A \to \B L(Y_0)$ given by
\[
\pi_0(a) = \begin{cases} \left(\begin{matrix} \pi_X(a) & 0 \\ 0 & \pi_X(a)\end{matrix}\right) & \text{for} \ a \text{ even}, \vspace{1mm}\\ \left(\begin{matrix} 0 & -\pi_X(a) \\ -\pi_X(a) & 0\end{matrix}\right) & \text{for} \ a \text{ odd} .\end{cases}
\]
By assumption $(2)$ we know that $(D_+ - D_-') \op (D_- - D_+')$ is a bounded perturbation of the unbounded selfadjoint and regular operator
\begin{gather*}
 ( P_+ \De P_+ + P_- \De' P_- ) \op ( - P_- \De P_- - P_+ \De' P_+)\colon \\
\qquad {}\big( \operatorname{Dom}( DP_+) \cap \operatorname{Dom}( D' P_-) \big) \op \big( \operatorname{Dom}( D'P_+) \cap \operatorname{Dom}( D P_-) \big) \to X \op X ,
\end{gather*}
where the upper diagonal entry and minus the lower diagonal entry are both unbounded positive and regular operators. But this shows that the fibre at $t = 0$ is a bounded perturbation of a spectrally decomposable unbounded Kasparov module. Indeed, the unbounded Kasparov module $( Y_0, ( P_+ \De P_+ + P_- \De' P_- ) \op ( - P_- \De P_- - P_+ \De' P_+) )$ from $\s A$ to $B$ is spectrally decomposable (using the orthogonal projection $\left(\begin{smallmatrix} 1 & 0 \\ 0 & 0\end{smallmatrix}\right)\colon X \op X \to X \op X$ when verifying the conditions in Definition \ref{d:spedec}).
\end{proof}

\begin{rem}It is worthwhile to understand the result of Proposition~\ref{p:phaequ} in the light of Skandalis' work in \cite{Ska:SRK}. This relationship was communicated to us by the referee. Indeed, under the assumptions of Proposition~\ref{p:phaequ} we obtain that $(X,F)$ is a Kasparov module from~$A$ to~$B$ (see the proof of the proposition). With some effort it can moreover be established that the assumptions imply that $a \big( F D\big(1 + D^2\big)^{-1/2} + D\big(1 + D^2\big)^{-1/2} F\big)a^* \geq 0$ modulo the compact operators on $X$ for all $a \in A$ and a similar result applies to $D'\big(1 + (D')^2\big)^{-1/2}$. Therefore, by \cite[Lemma~11]{Ska:SRK} we obtain that the bounded transforms $\big(X,D\big(1+D^2\big)^{-1/2}\big)$ and $\big(X,D'\big(1 + (D')^2\big)^{-1/2}\big)$ are operator homotopic since they are both operator homotopic to~$(X,F)$.

Notice however that this argument does not yield an alternative proof of Proposition~\ref{p:phaequ} since it only provides information on the bounded transforms of the unbounded Kasparov modules $(X,D)$ and $(X,D')$.
\end{rem}

\begin{Theorem}\label{t:group}The direct sum of unbounded Kasparov modules and the zero module provide the topological unbounded $KK$-theory, $UK^{{\rm top}}(\s A,B)$, with the structure of an abelian group.
\end{Theorem}
\begin{proof}For an unbounded Kasparov module $(X,D)$ from $\s A$ to $B$, we need to prove that $(X,D) - (X,D) = (X \op X^{{\rm op}}, D \op (-D))$ is stably homotopic to the zero module $(0,0)$. By Proposition \ref{p:lipinv} we may assume that $(X,D)$ is Lipschitz regular and that $D \colon \operatorname{Dom}(D) \to X$ is invertible. The phase of $D$ is then a well-defined odd selfadjoint unitary operator $F := D |D|^{-1}\colon X \to X$. The result of the present proposition will now be a consequence of Proposition~\ref{p:phaequ} applied to the case where $(X,D) = (X,D')$: The conditions $(1)$ and $(2)$ are clearly satisfied and condition $(3)$ follows from the Lipschitz regularity of $(X,D)$. Indeed, for each homogeneous $a \in \s A$ and each $\xi \in \operatorname{Dom}(D)$ it holds that
\[
[F,a](\xi) = |D|^{-1} [D,a](\xi) + (-1)^{\deg (a) + 1} |D|^{-1} [ |D|,a ] F(\xi)
\]
so that $[F,a] = |D|^{-1} T(a)$ for some bounded adjointable operator $T(a)\colon X \to X$.
\end{proof}

\section{Injectivity of the bounded transform}\label{s:injbou}
We are now ready to prove the main theorem of this paper:

\begin{Theorem}\label{t:isobaajul}
Suppose that $A$ and $B$ are a $\zz/2\zz$-graded $C^*$-algebras with~$A$ separable and~$B$ $\si$-unital. For any norm-dense countably generated $\zz/2\zz$-graded $*$-subalgebra $\s A \su A$ we have an isomorphism of abelian groups
\[
\C F\colon \ UK^{{\rm top}}(\s A, B) \to KK(A,B)
\]
induced by the Baaj--Julg bounded transform $(X,D) \mapsto \big[X, D\big(1 + D^2\big)^{-1/2}\big]$.
\end{Theorem}
\begin{proof}
By Theorems \ref{t:group} and~\ref{t:baajul} we only need to show that $\C F\colon UK^{{\rm top}}(\s A,B) \to KK(A,B)$ is injective.

Suppose that two unbounded Kasparov modules $(X_0,D_0)$ and $(X_1,D_1)$ from $\s A$ to $B$ satisfy that their bounded transforms $\big(X_0, D_0\big(1 + D_0^2\big)^{-1/2}\big)$ and $\big(X_1,D_1\big(1 + D_1^2\big)^{-1/2}\big)$ are homotopic. We thus have a Kasparov module $(X,F)$ from $A$ to $C([0,1],B)$ and two even unitary isomorphisms of $C^*$-correspondences $U_0\colon X \hot_{{\rm ev}_0} B \to X_0$ and $U_1\colon X \hot_{{\rm ev}_1} B \to X_1$ implementing unitary equivalences
\begin{gather*}
 \big(X_0, D_0 \big(1 + D_0^2\big)^{-1/2}\big) \sim_u \big(X \hot_{{\rm ev}_0} B, F \hot_{{\rm ev}_0} 1\big) \qquad \text{and} \\
 \big(X_1, D_1 \big(1 + D_1^2\big)^{-1/2}\big) \sim_u \big(X \hot_{{\rm ev}_1} B, F \hot_{{\rm ev}_1} 1\big) .
\end{gather*}
By Proposition \ref{p:liphom}, we may assume without loss of generality that $(X_0,D_0)$ and $(X_1,D_1)$ are both Lipschitz regular and, using \cite[Proposition~17.4.3]{Bla:KOA}, we may moreover assume that \mbox{$F\colon X \to X$} is a selfadjoint contraction. We let $X \op \wit X$ denote the $\zz/2\zz$-graded $C^*$-cor\-respon\-dence from $A$ to $C( [0,1],B)$ which agrees with $X \op X$ as a Hilbert $C^*$-module over $C([0,1],B)$ but with grading operator $\ga \op (-\ga)\colon X \op X \to X \op X$ and with left action given by the even $*$-homomorphism $\pi_{X \op \wit X} := \pi_X \op 0\colon A \to \B L\big(X \op \wit X\big)$. Define the odd bounded adjointable operator
\[
G := \left(\begin{matrix} F & \big(1 - F^2\big)^{1/2} \\ \big(1 - F^2\big)^{1/2} & - F\end{matrix}\right) \colon \ X \op \wit X \to X \op \wit X
\]
and notice that $G^2 = 1$ and $G = G^*$. Moreover, it holds that $\big(X \op \wit X, G\big)$ is a Kasparov module from $A$ to $C( [0,1],B)$ which is homotopic to $(X,F)$, see \cite[Section~17.6]{Bla:KOA}. A similar construction applies to the endpoints yielding Kasparov modules $\big(X_0 \op \wit{X_0}, G_0\big)$ and $\big(X_1 \op \wit{X_1}, G_1\big)$ from~$A$ to~$B$. For each $i \in \{0,1\}$, we define the invertible unbounded selfadjoint and regular operator
\[
E_i := \left(\begin{matrix} D_i & 1 \\ 1 & - D_i\end{matrix}\right) \colon \ \operatorname{Dom}(D_i) \op \operatorname{Dom}(D_i) \to X_i \op \wit X_i
\]
and remark that $G_i = E_i |E_i|^{-1}$. We recall from Proposition~\ref{p:lipinv} that $\big(X_i \op \wit{X_i}, E_i\big)$ is a~Lipschitz regular unbounded Kasparov module from $\s A$ to $B$ and that{\samepage
\[
\big[ X_i \op \wit{X_i}, E_i\big] = [X_i , D_i]
\]
in the topological unbounded $KK$-theory, $UK^{{\rm top}}(\s A,B)$.}

Using the Kasparov module $\big(X \op \wit X, G\big)$ from $A$ to $C([0,1],B)$ as input for Theorem~\ref{t:bousur}, we may choose an even strictly positive compact operator $\De\colon X \op \wit X \to X \op \wit X$ such that $(1)$, $(2)$, and $(3)$ in Theorem \ref{t:bousur} hold. In particular, we obtain an unbounded Kasparov module $\big( X \op \wit X, \De^{-1}G\big)$ from $\s A$ to $C([0,1],B)$ which implements a homotopy between the unbounded Kasparov modules
\[
\big(X_0 \op \wit{X_0}, \De_0^{-1} G_0\big) \qquad \text{and} \qquad \big(X_1 \op \wit{X_1}, \De_1^{-1} G_1\big) ,
\]
where by definition
\[
\De_i := (U_i \op U_i) \big(\De \hot_{{\rm ev}_i} 1 \big) (U_i^* \op U_i^*) \in \B K\big( X_i \op \wit X_i\big) , \qquad \text{for } i \in \{0,1\} .
\]

Let us fix an $i \in \{0,1\}$. Summarizing what has been obtained so far, we see that the proof of the theorem is finished provided that the identity
\[
\big[X_i \op \wit{X_i}, \De_i^{-1} G_i\big] = \big[ X_i \op \wit{X_i}, E_i\big]
\]
holds in $UK^{{\rm top}}(\s A,B)$. We are going to apply Proposition~\ref{p:phaequ} for our two unbounded Kasparov modules from $\s A$ to $B$ and the odd selfadjoint unitary operator $G_i\colon X_i \op \wit X_i \to X_i \op \wit X_i$. So we need to verify the three conditions in the statement of Proposition~\ref{p:phaequ}. For condition~$(1)$ we have that $G_i$ preserves the domains of $\De_i^{-1} G_i$ and $E_i$ and that both the non-graded commutators $\big[G_i, \De_i^{-1} G_i\big]\colon \operatorname{Dom}\big(\De_i^{-1} G_i\big) \to X_i \op \wit{X_i}$ and $[G_i, E_i]\colon \operatorname{Dom}(E_i) \to X_i \op \wit{X_i}$ are in fact trivial. For condition~$(2)$ we notice that
\begin{gather*}
 \De_i^{-1} G_i^2 = \De_i^{-1}\colon \ \operatorname{Dom}\big(\De_i^{-1}\big) \to X_i \op \wit{X_i} \qquad \text{and} \\
 E_i G_i = |E_i|\colon \ \operatorname{Dom}(|E_i|) \to X_i \op \wit{X_i}
\end{gather*}
are already even unbounded positive and regular operators (so no bounded perturbations are needed). To check the final condition $(3)$, we let $a \in \s A$ be homogeneous. The graded commutator $[G_i,a]$ has image contained in $\operatorname{Dom}\big(\De_i^{-1}\big) = \operatorname{Dom}\big(\De_i^{-1} G_i\big)$ and
\[
\De_i^{-1} G_i [G_i,a] = G_i \De_i^{-1} [G_i,a]\colon\ X_i \op \wit{X_i} \to X_i \op \wit{X_i}
\]
is bounded since $[G,a]\colon X \op \wit X \to X \op \wit X$ has image contained in $\operatorname{Dom}\big(\De^{-1}\big)$ and $\De^{-1} [G,a]$: $X \op \wit X \to X \op \wit X$ is bounded by construction, see Theorem~\ref{t:bousur}. The graded commutator $[G_i,a]$ also has image contained in $\operatorname{Dom}(E_i) = \operatorname{Dom}(|E_i|)$ since
\[
[G_i,a](\xi) = |E_i|^{-1} [E_i,a](\xi) + (-1)^{\deg (a) + 1} |E_i|^{-1}[|E_i|, a] |E_i|^{-1} E_i(\xi)
\]
for all $\xi \in \operatorname{Dom}(E_i)$ and since $\big(X_i \op \wit{X_i}, E_i\big)$ is Lipschitz regular. Letting $e_i(a)$ and $|e_i|(a)$ denote the bounded extensions of the graded commutators $[E_i,a]$ and $[|E_i|,a]\colon \operatorname{Dom}(E_i) \to X_i \op \wit{X_i}$ we moreover see that
\[
E_i [G_i,a] = G_i \cd e_i(a) + (-1)^{\deg (a) + 1} G_i \cd |e_i|(a) \cd G_i\colon X_i \op \wit{X_i} \to X_i \op \wit{X_i}
\]
is bounded. It thus follows from Proposition \ref{p:phaequ} and Theorem \ref{t:group} that $\big[X_i \op \wit{X_i}, \De_i^{-1} G_i\big] = \big[X_i \op \wit{X_i}, E_i\big]$ in $UK^{{\rm top}}(\s A,B)$ and this ends the proof of the theorem.
\end{proof}

\subsection*{Acknowledgements}
The starting point for this paper was a couple of conversations with Bram Mesland during the thematic programme on ``Bivariant $K$-theory in Geometry and Physics'' at the Erwin Schr\"odinger Institute in Vienna in November 2018. I would like to thank the ESI for their hospitality and support and the organizers of the thematic programme for this great opportunity to meet and discuss the unbounded approach to $KK$-theory and its applications in mathematical physics. As always, I am also grateful to my friends and collaborators Magnus Goffeng, Bram Mesland, and Adam Rennie for many good conversations on unbounded $KK$-theory and its relationship to $KK$-theory. Finally, I would like to thank the anonymous referee for her/his comments regarding spectral decomposability and its relationship to the work of Skandalis.

The author was partially supported by the DFF-Research Project 2 ``Automorphisms and Invariants of Operator Algebras'', no.~7014-00145B.

\pdfbookmark[1]{References}{ref}
\LastPageEnding


\begin{thebibliography}{99}
\footnotesize\itemsep=0pt

\bibitem{Ati:GEO}
Atiyah M.F., Global theory of elliptic operators, in Proc. {I}nternat. {C}onf.
 on {F}unctional {A}nalysis and {R}elated {T}opics ({T}okyo, 1969), University
 of Tokyo Press, Tokyo, 1970, 21--30.

\bibitem{BaJu:TBK}
Baaj S., Julg P., Th\'eorie bivariante de {K}asparov et op\'erateurs non
 born\'es dans les {$C^*$}-modules hilbertiens, \textit{C.~R.~Acad. Sci. Paris
 S\'er.~I Math.} \textbf{296} (1983), 875--878.

\bibitem{Bla:KOA}
Blackadar B., {$K$}-theory for operator algebras, 2nd~ed., \textit{Mathematical Sciences
 Research Institute Publications}, Vol.~5, Cambridge University
 Press, Cambridge, 1998.

\bibitem{BrMeSu:GSU}
Brain S., Mesland B., van Suijlekom W.D., Gauge theory for spectral triples and
 the unbounded {K}asparov product, \href{https://doi.org/10.4171/JNCG/230}{\textit{J.~Noncommut. Geom.}} \textbf{10}
 (2016), 135--206, \href{https://arxiv.org/abs/1306.1951}{arXiv:1306.1951}.

\bibitem{BDF:ECK}
Brown L.G., Douglas R.G., Fillmore P.A., Extensions of {$C^*$}-algebras and
 {$K$}-homology, \href{https://doi.org/10.2307/1970999}{\textit{Ann. of Math.}} \textbf{105} (1977), 265--324.

\bibitem{Con:NCG}
Connes A., Noncommutative geometry, Academic Press, Inc., San Diego, CA, 1994.

\bibitem{Con:GCM}
Connes A., Gravity coupled with matter and the foundation of non-commutative
 geometry, \href{https://doi.org/10.1007/BF02506388}{\textit{Comm. Math. Phys.}} \textbf{182} (1996), 155--176,
 \href{https://arxiv.org/abs/hep-th/9603053}{arXiv:hep-th/9603053}.

\bibitem{DGM:BGU}
Deeley R.J., Goffeng M., Mesland B., The bordism group of unbounded
 {$KK$}-cycles, \href{https://doi.org/10.1142/S1793525318500012}{\textit{J.~Topol. Anal.}} \textbf{10} (2018), 355--400,
 \href{https://arxiv.org/abs/1503.07398}{arXiv:1503.07398}.

\bibitem{Hil:BIK}
Hilsum M., Bordism invariance in {$KK$}-theory, \href{https://doi.org/10.7146/math.scand.a-15143}{\textit{Math. Scand.}}
 \textbf{107} (2010), 73--89.

\bibitem{JeTh:EKT}
Jensen K.K., Thomsen K., Elements of {$KK$}-theory, \textit{Mathematics: Theory \&
 Applications}, \href{https://doi.org/10.1007/978-1-4612-0449-7}{Birkh\"auser Boston, Inc.}, Boston, MA, 1991.

\bibitem{Kaa:MEU}
Kaad J., Morita invariance of unbounded bivariant {$K$}-theory,
 \href{https://arxiv.org/abs/1612.08405}{arXiv:1612.08405}.

\bibitem{Kaa:DAH}
Kaad J., Differentiable absorption of {H}ilbert {$C^*$}-modules, connections,
 and lifts of unbounded operators, \href{https://doi.org/10.4171/JNCG/11-3-8}{\textit{J.~Noncommut. Geom.}} \textbf{11}
 (2017), 1037--1068, \href{https://arxiv.org/abs/1407.1389}{arXiv:1407.1389}.

\bibitem{KaLe:LGR}
Kaad J., Lesch M., A local global principle for regular operators in {H}ilbert
 {$C^*$}-modules, \href{https://doi.org/10.1016/j.jfa.2012.03.002}{\textit{J.~Funct. Anal.}} \textbf{262} (2012), 4540--4569,
 \href{https://arxiv.org/abs/1107.2372}{arXiv:1107.2372}.

\bibitem{KaLe:SFU}
Kaad J., Lesch M., Spectral flow and the unbounded {K}asparov product,
 \href{https://doi.org/10.1016/j.aim.2013.08.015}{\textit{Adv. Math.}} \textbf{248} (2013), 495--530, \href{https://arxiv.org/abs/1110.1472}{arXiv:1110.1472}.

\bibitem{KaSu:FDT}
Kaad J., van Suijlekom W.D., Factorization of {D}irac operators on toric
 noncommutative manifolds, \href{https://doi.org/10.1016/j.geomphys.2018.05.027}{\textit{J.~Geom. Phys.}} \textbf{132} (2018),
 282--300, \href{https://arxiv.org/abs/1803.08921}{arXiv:1803.08921}.

\bibitem{Kas:TIE}
Kasparov G.G., Topological invariants of elliptic operators. {I}.
 {$K$}-homology, \href{https://doi.org/10.1070/IM1975v009n04ABEH001497}{\textit{Math. USSR-Izv.}} \textbf{9} (1975), 751--792.

\bibitem{Kas:HSV}
Kasparov G.G., Hilbert {$C^*$}-modules: theorems of {S}tinespring and
 {V}oiculescu, \textit{J.~Operator Theory} \textbf{4} (1980), 133--150.

\bibitem{Kas:OFE}
Kasparov G.G., The operator {$K$}-functor and extensions of {$C^*$}-algebras,
 \href{http://dx.doi.org/10.1070/IM1981v016n03ABEH001320}{\textit{Math. USSR-Izv.}} \textbf{16} (1980), 513--572.

\bibitem{Kuc:PUM}
Kucerovsky D., The {$KK$}-product of unbounded modules, \href{https://doi.org/10.1023/A:1007751017966}{\textit{$K$-Theory}}
 \textbf{11} (1997), 17--34.

\bibitem{Kuc:LIK}
Kucerovsky D., A lifting theorem giving an isomorphism of {$KK$}-products in
 bounded and unbounded {$KK$}-theory, \textit{J.~Operator Theory} \textbf{44}
 (2000), 255--275.

\bibitem{Lan:HCM}
Lance E.C., Hilbert {$C^*$}-modules. A toolkit for operator algebraists,
 \textit{London Mathematical Society Lecture Note Series}, Vol.~210, \href{https://doi.org/10.1017/CBO9780511526206}{Cambridge
 University Press}, Cambridge, 1995.

\bibitem{LeMe:SSH}
Lesch M., Mesland B., Sums of regular self-adjoint operators in
 {H}ilbert-{$C^*$}-modules, \href{https://doi.org/10.1016/j.jmaa.2018.11.059}{\textit{J.~Math. Anal. Appl.}} \textbf{472} (2019),
 947--980, \href{https://arxiv.org/abs/1803.08295}{arXiv:1803.08295}.

\bibitem{Mes:UCN}
Mesland B., Unbounded bivariant {$K$}-theory and correspondences in
 noncommutative geometry, \href{https://doi.org/10.1515/crelle-2012-0076}{\textit{J.~Reine Angew. Math.}} \textbf{691} (2014),
 101--172, \href{https://arxiv.org/abs/0904.4383}{arXiv:0904.4383}.

\bibitem{MeRe:NST}
Mesland B., Rennie A., Nonunital spectral triples and metric completeness in
 unbounded {$KK$}-theory, \href{https://doi.org/10.1016/j.jfa.2016.08.004}{\textit{J.~Funct. Anal.}} \textbf{271} (2016),
 2460--2538, \href{https://arxiv.org/abs/1502.04520}{arXiv:1502.04520}.

\bibitem{Ska:SRK}
Skandalis G., Some remarks on {K}asparov theory, \href{https://doi.org/10.1016/0022-1236(84)90081-8}{\textit{J.~Funct. Anal.}}
 \textbf{56} (1984), 337--347.

\bibitem{DuMe:HEU}
van~den Dungen K., Mesland B., Homotopy equivalence in unbounded {$KK$}-theory,
 \href{https://doi.org/10.2140/akt.2020.5.501}{\textit{Ann. K-Theory}} \textbf{5} (2020), 501--537, \href{https://arxiv.org/abs/1907.04049}{arXiv:1907.04049}.

\bibitem{Wor:UAQ}
Woronowicz S.L., Unbounded elements affiliated with {$C^*$}-algebras and
 noncompact quantum groups, \href{https://doi.org/10.1007/BF02100032}{\textit{Comm. Math. Phys.}} \textbf{136} (1991),
 399--432.

\bibitem{WoNa:OTC}
Woronowicz S.L., Napi\'orkowski K., Operator theory in the {$C^*$}-algebra
 framework, \href{https://doi.org/10.1016/0034-4877(92)90025-V}{\textit{Rep. Math. Phys.}} \textbf{31} (1992), 353--371.

\end{thebibliography}
\end{document}